\documentclass{article}

\usepackage{amsfonts,amssymb, latexsym, amsmath, amsthm,mathrsfs, verbatim,geometry}

\usepackage{slashed}
\usepackage{amsfonts}
\usepackage{url,color}
\usepackage{enumerate}
\usepackage{esint}
\usepackage{pifont}
\usepackage{upgreek}
\usepackage{times}
\usepackage{calligra}
\usepackage{graphicx}
\usepackage{caption}
\usepackage{amscd}
\usepackage{amsthm}
\usepackage{amsmath,bm}
\usepackage{amssymb}
\usepackage{comment}
\usepackage{tikz}
\usepackage[utf8]{inputenc}
\usepackage{gensymb}
\usepackage{pifont}
\usepackage{autobreak}
\usepackage{color}
\usepackage{verbatim}
\usepackage{systeme}
\usepackage[font=scriptsize]{caption}
\usepackage[ngerman, english]{babel}
\usepackage[title,toc,page]{appendix}
\usepackage{subfig}
\usepackage{tikz}
\usetikzlibrary{calc}
\usetikzlibrary{calc}
\usetikzlibrary{arrows.meta}
\usetikzlibrary{hobby}
\usepackage{tkz-euclide}

\newcommand{\ga}{\gamma}
\newcommand{\la}{\lambda}

\def\be{\begin{equation}}
\def\ee{\end{equation}}

\numberwithin{equation}{section}

\usepackage{listings}

\allowdisplaybreaks

\usepackage{hyperref}
\hypersetup{
    colorlinks,
    citecolor=red,
    filecolor=green,
    linkcolor=blue,
    urlcolor=black
}
\definecolor{backcolour}{rgb}{0.95,0.95,0.92}
\lstdefinestyle{mystyle}{
    backgroundcolor=\color{backcolour},   
}

\usepackage{tabularx}
\usepackage{multirow}
\usepackage{blindtext} 
\usepackage{charter} 
\usepackage{microtype} 
\usepackage[english, ngerman]{babel} 
\usepackage{amsthm, amsmath, amssymb} 
\usepackage{float} 
\usepackage{graphicx, multicol} 
\usepackage{xcolor} 
\usepackage{marvosym, wasysym} 
\usepackage{rotating} 
\usepackage{censor} 
\usepackage{pseudocode} 
\usepackage{booktabs} 
\usepackage{csquotes} 
\newtheorem{theorem}{Theorem}[section]

\newtheorem{lemma}[theorem]{Lemma}
\newtheorem{remark}[theorem]{Remark}

\title{Waiting Time Solutions in gas dynamics}

\author{Juhi Jang\thanks{Department of Mathematics, University of Southern California, Los Angeles, CA 90089, USA, and Korea Institute for Advanced Study, Seoul, Korea.  Email: juhijang@usc.edu.}, \ Jiaqi Liu\thanks{Department of Mathematics, University of Southern California, Los Angeles, CA 90089, USA, Email: jiaqil@usc.edu}, \ and \ Nader Masmoudi\thanks{ NYUAD Research Institute, New York University Abu Dhabi, Abu Dhabi, UAE,
and Courant Institute of Mathematical Sciences, New York University, 251 Mercer St, New York,
NY 10012, USA, Email: nm30@nyu.edu}}

\date{}
\begin{document}

\maketitle
\abstract{In this article, we construct a continuum family of self-similar waiting time solutions for the one-dimensional compressible Euler equations for the adiabatic exponent $\ga\in(1,3)$ in the half-line with the vacuum boundary. 
The solutions are confined by a stationary vacuum interface for a finite time with at least $C^1$  regularity of the velocity and the sound speed up to the boundary. Subsequently, the solutions  
undergo the change of the behavior, becoming only H\"{o}lder continuous near the singular point, and simultaneously transition to the solutions to the vacuum moving boundary Euler equations satisfying the physical vacuum condition.  When the boundary starts moving, a weak discontinuity emanating from the singular point along the sonic curve emerges. The solutions are locally smooth in the interior region away from the vacuum boundary and the sonic curve.  
}

\section{Introduction}

We consider one-dimensional polytropic gases governed by compressible Euler equations:
\be\label{Euler}
\begin{split}
\rho_t+ (\rho u)_x &=0, \\
(\rho u)_t+ (\rho u^2 +p)_x &= 0. 
\end{split}
\ee
Here $\rho:(-T,T)\times \Omega \to \mathbb R_{\ge0}$, $u:(-T,T)\times \Omega \to \mathbb R$, $p:(-T,T)\times \Omega \to \mathbb R_{\ge0}$ represent the density, velocity, and pressure of the gas, and $\Omega$ is the spatial domain given by non-vacuum region. The pressure obeys the $\gamma$-law
\be\label{pressure}
p=A \rho^\gamma
\ee
where $A$, $\gamma$ are positive constants so that $A>0$ and $1<\gamma< 3$. Fix $A$ so that $\frac{A\gamma}{\gamma-1}=1$ and introduce the specific enthalpy denoted by $h$
\be\label{enthalpy}
h= \rho^{\gamma-1}. 
\ee
With this normalization, the speed of the sound $c= \sqrt{\frac{dp}{d\rho}} $ is given by 
\[
c= \sqrt{(\gamma-1)h}
\]
and in non-vacuum region, the Euler system \eqref{Euler} can be equivalently written as 
\be\label{Eulerh}
\begin{split}
h_t + u h_x +(\gamma-1) h u_x &=0,\\
u_t + u u_x + h_x&= 0. 
\end{split}
\ee
In the case of $\gamma=2$, $h$ and $\rho$ are the same (cf. \eqref{enthalpy}) and the system $\eqref{Eulerh}_{\gamma=2}$ is also known as the shallow water equations, where $h$ represents the height of water level and $u$ the velocity. 

In this article, we are interested in the boundary behavior of  solutions to \eqref{Euler} when they are continuously in contact with vacuum region. Specifically, we consider the Euler system \eqref{Euler} in the  spatial domain  
$$\Omega (t) := \{ b(t) < x <\infty \}$$ 
representing the support of the fluid. The left boundary $b(t)$ is the first point of contact to vacuum state, initially set to be 0, satisfying the following vacuum boundary conditions 
\be\label{BC1}
\rho (t,x) >0 \ \text{ for } \ b(t)<x<\infty, \quad \rho (t, b(t))=0
\ee
and 
\be\label{BC:kinematic}
 \dot b(t) = u (t, b(t)), \ \ b(-T) = 0 , 
\ee
with initial data 
\be\label{IC}
\rho(-T,x)= \rho_0(x), \ \  u(-T,x)=u_0(x) \text{ for } x>0
\ee 
where 
\[
  \ \rho_0(x)>0  \text{ for } x>0, \text{ and } \  \rho_0(0)=0. 
\]
The boundary point $b(t)$ can be viewed as the interface between fluid region and vacuum region and $x=0$ is the initial interface.

It is well-known that in order for such a vacuum initial boundary value problem to be well-posed in reasonable function spaces even locally-in-time, additional regularity or stability condition on the normal derivative of the specific enthalpy $h$ is required at initial time and to be propagated in time. One important example is the so-called physical vacuum boundary, which can be phrased in the current setting as 
\be\label{PVB}
0<\frac{\partial h}{ \partial x}(t,b(t)) <\infty , 
\ee
which represents a nontrivial finite acceleration near the vacuum boundary. 
This physical vacuum boundary condition arises in various contexts ranging from gas dynamics to astrophysics when studying the dynamics of isolated compactly supported bodies \cite{CoSh2012, Jang2014, JM, LXZ, Sideris17}. There has been a lot of progress over the past decade on the local and global behavior of solutions for the vacuum free boundary problems for compressible fluids, for instance see  \cite{CoSh2012, GHJ2021a, HaJa2018-1, HaJa2016-2, HaJa2017,  IfTa2020, JaMa2009, JaMa2015}. 

A natural question is what happens if the physical vacuum condition does not hold initially. Of course, if the velocity is zero initially at the boundary $u_0(0)=0$,  in view of \eqref{BC:kinematic},  one may be able to propagate the zero velocity along the vacuum boundary and expect a stationary interface at least for some time. However, such a propagation is not always possible. In \cite{JM},  the first and third authors studied well-posedness and ill-posedness questions of various vacuum states depending on different behavior of $\frac{\partial h_0}{\partial x} \sim x^\alpha $.  In particular, it was shown that when $\alpha=1$ (see \cite{LY} for local-in-time well posedness), the boundary behavior $\frac{\partial h}{\partial x}\sim x $ may not persist for long in general, and conjectured that after some time the boundary should start moving and attain the physical vacuum \eqref{PVB}, and for other ranges of $\alpha$ an instantaneous change of behavior into the physical vacuum might occur, see also \cite{LY1}. In fact,  an interesting recent work  \cite{Jenssen} of Jenssen shows that some self-similar solutions to  \eqref{Euler} exhibit an instantaneous change of the vacuum boundary, namely the vacuum boundary moves immediately even if $u_0(0)=0$, while some other self-similar solutions are confined by the stationary interface $x=0$.  

The goal of this article is to show the existence of waiting time solutions with $u_0, \frac{\partial h_0}{\partial x} \sim_{x\ll 1} x$ of which boundary moves after a finite time satisfying the physical vacuum \eqref{PVB}.  We obtain these solutions as a family of self-similar solutions to the free boundary Euler equations  \eqref{Euler}, \eqref{BC1}-\eqref{IC}, see Theorem \ref{main theorem} for the detailed description.  

We mention the work  \cite{GB}, where the authors gave some asymptotic analysis alluding waiting time solutions and instantaneous change near the boundary $x=0$ and claimed the existence of self-similar solutions of the shallow water equation, corresponding to $\gamma=2$. 
And in \cite{Camassa20}, the authors investigated the vacuum dam-break problem with 
 constant far-field initial conditions and constructed various solutions including 
 a non-smooth waiting-time solution. We also remark that the change of behavior or the waiting time behavior of certain vacuum states have been thoroughly studied for nonlinear diffusion equations such as the porous medium equation, for instance see \cite{CF1979, Knerr1977}. 

 
\section{Self-similar ODEs and the main result}

In this section, we derive self-similar ODEs and state the main result. A particular family of self-similar solutions 
will play an important role in our construction. To motivate them, we first discuss simple wave solutions to the Euler equations  \eqref{Euler} and connections to Burgers' equation. 

\subsection{Riemann invariants and simple wave solutions}\label{Sec:2.1}

The Euler equations \eqref{Euler} have the Riemann invariants: 
\[
		z = u - \tfrac{2}{\gamma-1}  c \  \text{and} \ w= u + \tfrac{2}{\gamma-1} c, 
\]
which diagonalize the system: these Riemann invariants $(z,w)$ satisfy the transport equations 
	\begin{align}\label{zw}
		z_t + \la_1 z_x &=0,\\
		w_t+\la_2 w_x &=0,\notag
	\end{align}
	where $\la_1  = u - c$ and $\la_2 = u +c$. The system \eqref{zw} admits the so-called {\it simple wave} solutions for which one of Riemann invariants is constant, for instance see \cite[p164]{Whitham1999}. In particular, if we set $w\equiv 0$
	such that $ u=-\frac{2}{\gamma-1}c$, the $z$-equation leads to  
	the Burgers' equation for $u$
	\be \label{Burgers}
	u_t + \tfrac{\gamma+1}{2} u u _x =0 . 
	\ee	
	If $u$ is a non-positive solution to \eqref{Burgers} for $x\ge 0$ with $u(t,0)=0$, then $u$ and $c=- \frac{\gamma-1}{2}u$ solve the Euler equations \eqref{zw} or  \eqref{Euler} in the half-line. 
	
	On the other hand, it is well-known that classical solutions to Burgers' equation \eqref{Burgers} with an initially negative slope will break down in a finite time and continue as discontinuous shock wave (weak) solutions. After the singularity forms, Burgers' shocks do not serve as simple waves to the Euler equations any longer, since the corresponding simple wave pairs in general do not satisfy the Rankine-Hugoniot conditions for the Euler equations \eqref{Euler}. 
	
	
	In this paper, we show that simple wave solutions emanating from self-similar Burgers solutions, after the first blowup time, {\it continue} as H\"{o}lder continuous solutions to the vacuum free boundary problem whose vacuum boundary satisfies the physical vacuum condition \eqref{PVB}. To this end, we next introduce the self-similarity and derive the self-similar reduction for the Euler system \eqref{Eulerh}.\footnote{We work with \eqref{Eulerh} for $h$ and $u$, as it is convenient for the appearance of physical vacuum.}

\subsection{Scaling invariance and self-similarity as $x\rightarrow 0$ and $t\rightarrow 0$}

The system \eqref{Eulerh} admits a two-parameter family of invariant scalings: the scaling transformation
\be\label{scaling}
u (t,x) \to \nu^{\delta -1 } u ( \frac{t}{\nu},\frac{x}{\nu^\delta }), \quad h (t,x) \to  \nu^{2(\delta -1) } h ( \frac{t}{\nu},\frac{x}{\nu^\delta })
\ee
for $\nu, \, \delta>0$ leaves the system invariant. This scaling symmetry is closely tied to the existence of self-similar solutions. Self-similarity is an important concept in hydrodynamics due to its universal nature and the possibility that self-similar solutions are attractors for different physical phenomena in fluid and gas dynamics \cite{CGN18, EF00, Guderley42, Landau87, Stanyukovich60}. 

Inspired by the scaling symmetry \eqref{scaling}, we introduce the similarity variable: 
\be
y = \frac{x}{|t|^\delta }, \quad \delta >0
\ee
for $t\in \mathbb R$. We then look for solutions of \eqref{Eulerh} which are valid locally as $x\rightarrow 0$ and $t\rightarrow 0$ from above or below where $(t,x)=(0,0)$ corresponds to the first singularity point.\footnote{By translation symmetry of the system, we may consider $\tau = t- T$ to shift the first time of singularity.} 
If $t<0$  (before the singularity), let 
\be\label{2.14}
u (t,x)= - \delta (-t)^{\delta -1} y U(y), \quad h{(t,x)} = \delta^2 (-t)^{2(\delta-1)} y^2 \frac{H(y)}{\gamma-1}, \quad y= \frac{x}{(-t)^\delta}, \quad \delta>0, 
\ee
and for $t>0$ (after the singularity), let 
\be\label{2.15}
u (t,x)= \delta t ^{\delta-1} y U(y), \quad h{(t,x)}  = \delta^2 t^{2(\delta-1)} y^2 \frac{H(y)}{\gamma-1}, \quad y = \frac{x}{t^\delta}, \quad \delta >0. 
\ee

For $t<0$, $x=0$ is a boundary point and $x>0$ is an interior point. Hence $y\ge 0$ in \eqref{2.14}. To observe the change of boundary behavior, we expect $u(t,x)\le 0$ near $x=0$. Therefore, in \eqref{2.14}, we demand $U(y)\geq 0$ and $H(y)\geq 0$ for each $y\ge 0$. On the other hand, when $t>0$,  we expect a moving interface, so $y$ may take either sign. In \eqref{2.15}, we require $U(y)>0 $ for $y<0$, while we have $U(y)<0$  for $y>0$. In both cases $H(y)\ge 0$. 

\begin{remark}  We demand the same values of $\delta$  in \eqref{2.14} and \eqref{2.15}. In fact, a similar argument in Lemma \ref{PassA} (cf. \eqref{weak sol}) shows that in order for the extension in $x>0$ from $t=0^-$ to $0^+$ to yield weak solutions to the Euler equations, the same exponents must be chosen. 
\end{remark}	
	


\subsection{Self-similar ODE}

With the self-similar ansatz \eqref{2.14} and \eqref{2.15}, the equations \eqref{Eulerh} are reduced to the following the system of ODEs in both $t<0$ and $t>0$ 
\begin{align}\label{ode system}
\begin{cases}y\cfrac{dH}{dy} &= \cfrac{2H[H-(U^2-k_1U+\mu)]}{(U-1)^2-H} =:\cfrac{F(H,U,\ga,\mu)}{\Delta(H,U)},\\
y\cfrac{dU}{dy} &= \cfrac{H(U+k_2)-U(U-1)(U-\mu)}{(U-1)^2-H}=:\cfrac{G(H,U,\ga,\mu)}{\Delta(H,U)}\end{cases}
\end{align}
where 
\begin{align}
	k_1 &= \frac{(\ga+1)+\mu(3-\ga)}{2}, \quad	k_2  = \frac{2(1-\mu)}{\ga-1}, \quad \mu = \frac{1}{\delta} \label{k2}.
\end{align}
For notational convenience, we will use both $\delta$ and $\mu=\frac{1}{\delta}$ throughout the paper. 

We will also consider the effective ODE for $H=H(U)$: 
\begin{align}\label{ODE}
	\frac{dH}{dU} = \frac{F(H,U,\ga,\mu)}{G(H,U,\ga,\mu)}. 
\end{align} 

The ODE system \eqref{ode system} has possible singularities either when $y=0$ or when $\Delta =0 $. The latter is related to the so-called sonic points. We define 

\

{\bf Sonic curves:} $H= (U-1)^2$, which is the set where $\Delta(H,U)=0$. 

\

The sonic points have played an important role in the recent constructions of smooth profiles for the implosion, gravitational collapse and of collapsing-reflected shocks \cite{GHJ21, GHJ23, GHJS22, JLS23, JLS24, Jenssen, Jenssen18, Jenssen23, Merle22a}. If the trajectory is inevitably passing through the sonic point, it has to pass smoothly, in other words, $F=G=\Delta=0$ at the sonic point. Many of recent constructions of self-similar singular solutions do enjoy smoothness, and in some cases, regularity acts as a selection criterion for the solution. If it can't pass smoothly, the flow might undergo the jump beforehand and continue as a discontinuity (shock); in this case, the sonic point can be thought of as the maximal development of the flow, and the corresponding solutions may be weak solutions such as physical shock discontinuities, for instance as in the classical reflected shocks of {Guderley} \cite{Guderley42, JLS24}. In our problem, the sonic points will play a pivotal role in the continuation part after the first singularity occurs. It turns out that the trajectory can pass the first  sonic point Lipschitz-continuously, therefore exhibiting weak discontinuities, and it will stop at the second sonic point representing the vacuum boundary.

We note that the effective ODE \eqref{ODE} admits the special solution ${H}^{\text{sp}}(U) = \frac{(\ga-1)^2}{4}U^2$.

\begin{lemma}\label{L: special solution}
	For any $\ga\in(1,3)$ and $\mu\in(0,1)$, $H^{\text{sp}}(U) = \frac{(\ga-1)^2}{4}U^2$ is always a solution to the ODE \eqref{ODE}.
\end{lemma}

\begin{proof}
	By plugging $H(U) = \frac{(\ga-1)^2}{4}U^2$ into $ \frac{F(H,U,\ga,\mu)}{G(H,U,\ga,\mu)}$, it is clear that $H'(U) = \frac{(\ga-1)^2}{2}U = \frac{dH}{dU}|_{H=H(U)}$.
\end{proof}

The special solution $H^{\text{sp}}(U) = \frac{(\ga-1)^2}{4}U^2$ is indeed a manifest of the simple wave solution satisfying $u=-\frac{2}{\gamma-1}c$ discussed in Section \ref{Sec:2.1}. The ODE system \eqref{ode system} associated with this special solution leads to the single ODE (see \eqref{ode B}), which is the self-similar ODE for the Burgers' equation.\footnote{The equation for $- yU(y)$ in \eqref{ode B} will lead to the form frequently used in other works \cite{CGN18, EF00}.} We will take this special solution as our solution until it meets the sonic curve $H= (U-1)^2$.

We are now ready to state the main result of the paper. 

\begin{theorem}\label{main theorem} Let $\gamma\in (1,3)$ be given. For each $\mu\in (0,1)$, there exist a continuum family of $\rho$ and $u$ solving the vacuum free boundary Euler  \eqref{Euler}, \eqref{BC1}-\eqref{IC} with the left vacuum boundary $b(t)$:  
\be\label{b(t)}
b(t) = 
\begin{cases}
0, & t<0  \\
y_B t^{\frac{1}{\mu}} & t\ge 0 
\end{cases}
\ee
for some $y_B <0$. 
Moreover, $(\rho,u)$ satisfy the following: 
\begin{enumerate}
\item 
They are self-similar, namely  of the form \eqref{2.14} for $t<0$ and \eqref{2.15} for $t>0$. 
\item There exist $-\infty<y_B<y_D<0$ such that they are smooth in the regions $\{(t,x): x>0, \  t\in \mathbb R\} $,  $ \{(t,x):   y_D t^{\delta} < x , \  t>0\}$,  $ \{(t,x): y_B t^{\delta} < x < y_D t^\delta, \  t>0    \}$ (see Figure \ref{fig:tr}), but only Lipschitz continuous across the sonic curve $x= y_D t^\delta$, $t>0$. 
\item They are H\"{o}lder continuous up to the singular point $\mathcal O=(0,0)$. 
\item They satisfy the physical vacuum condition \eqref{PVB} for each $t>0$. 
\end{enumerate}
\end{theorem}


\begin{figure}
\captionsetup{width=0.8\textwidth}
\begin{center}
\begin{tikzpicture}[scale=1.2,domain=-4:4]
  \draw[->,line width=0.25mm] (0,0) -- (2,0) node[right,scale=0.72] {$x$};
  \draw[->,line width=0.25mm] (0,0) -- (0,2) node[above,scale=0.72] {$t$};
    \draw[line width=0.25mm] (0,0) -- (0,-1.8);
        \draw[line width=0.25mm] (0,0) -- (-2,0);
 \node[right,color=black,scale=0.7](Q) at (0,-1.6) {$y_C=0$};
    \node[above,color=black,scale=0.7](Q) at (1.5,0) {$ y_A=+\infty$};
    \node[below,color=black,scale=0.7](Q) at (1.5,0) {$y_A=+\infty$};
    \node[right,color=black,scale=0.7](Q) at (0,0) {$\mathcal O$};
    \draw[smooth,color=blue,domain=-1.5:0] plot (\x,{sqrt(-\x)});
    \node[above,color=black,scale=0.7](Q) at (-1.8,1.1) {$y=y_D$};
     \draw[smooth,color=red,domain=-1.5:0] plot (\x,{0.5*sqrt(-\x)});
    \node[above,color=black,scale=0.7](Q) at (-1.8,0.5) {$y=y_B$};
    \node[above,color=black,scale=0.7](Q) at (-1.2, 0.05) {Vacuum Region};
    \node[right,color=black,scale=0.7](Q) at (0,1.8) {$y_E = 0^+$};
    \node[left,color=black,scale=0.7](Q) at (0,1.8) {$y_E = 0^-$};
        \node[above,color=black,scale=0.7](Q) at (-1.2, -0.8) {Vacuum Region};
\end{tikzpicture}
 \end{center}
 \caption{The structure of the waiting time solution. The red line shows the moving boundary $\{ y \equiv y_B \}=\{x=y_B t^\delta\}$. 
 The blue line corresponds to the sonic curve $\{ y \equiv y_D \}=\{x=y_D t^\delta\}$ representing the weak discontinuity.
   }\label{fig:tr}
    \vspace{-5mm}
\end{figure}
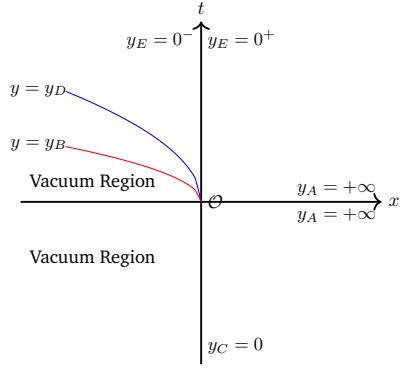

\smallskip 

To the best of our knowledge, Theorem \ref{main theorem} provides the first rigorous construction of waiting time solutions in gas dynamics exhibiting a stationary vacuum interface changes into a moving interface with the physical vacuum after a finite time. As described above, these solutions are locally smooth away from the boundary and away from the sonic curve, and the solutions are H\"{o}lder continuous up to the singular point. The sonic curve emanating from the singular point $\mathcal O=(0,0)$ represents a weak discontinuity: the derivatives of the velocity and enthalpy undergo a jump across the sonic curve.

The proof of Theorem \ref{main theorem} is based on a careful study of the self-similar ODE system \eqref{ode system} and the extension of the trajectory  through critical points. We take the trajectory induced by the special solution $H^{\text{sp}}(U) = \frac{(\ga-1)^2}{4}U^2$ for $x>0$ and $t<0$, and extend it smoothly via the special solution $H^{\text{sp}}(U) = \frac{(\ga-1)^2}{4}U^2$ for $x>y_D t^\delta$ and $t>0$, namely until it meets the first sonic point (Section \ref{Sec 3}). In order to extend the trajectory  reaching the vacuum boundary, it cannot pass through the sonic point by following the special solution. A detailed analysis around the sonic point using the Poincar\'{e}-Dulac Theorem reveals that there exists another direction allowed for the solution to pass through the sonic point and to arrive at the vacuum boundary, the second sonic point. To make the connection, we construct the unique trajectory starting from the second sonic point, the vacuum boundary represented by the curve $\{x=y_B t^\delta\}$, to the first sonic point represented by $\{x=y_D t^\delta\}$ by a barrier argument (Section \ref{Sec 4}). From the construction, it is evident that the solutions are not smooth across the sonic curve. In Section \ref{Sec 5}, the proof of Theorem \ref{main theorem} is completed. 

We conclude this section with a couple of remarks on our waiting time solutions in Theorem \ref{main theorem}. 
First of all, for each $\gamma\in (1,3)$, we obtain a two-parameter continuum family of the solutions; one parameter is the self-similar exponent $\mu\in (0,1)$ ($\delta>1$) and the other freedom is coming from the arbitrary coefficient of the next order expansion of the initial profiles (cf. Remark \ref{remark 2.3} and \eqref{(3.22)}). In fact, for given  $\gamma\in (1,3)$ and $\mu\in (0,1)$, the solution is unique modulo such a freedom determined by the coefficient, in the class of self-similar solutions considered in this article. We only consider $\mu\in (0,1)$ as other values of $\mu$ are not relevant to our study (cf. Remark \ref{remark3.2}). 
 We note that $\mu\in (0,1)$ directly enters into the regularity of the solutions near the boundary and the boundary geometry (cf. Remark \ref{remark 2.3}), and in particular,  if $\frac{1}{1-\mu}\in \mathbb N$, the initial regularity of the profile is smooth up to the boundary (cf. Remark \ref{remark_initial}). 

Self-similar waiting time solutions constructed in Theorem \ref{main theorem} are relevant in the neighborhood of $\mathcal O=(0,0)$ undergoing the change of behavior and in turn allowing the continuation of the solutions; by the finite speed of propagation enjoyed by Euler flows and the cutoff argument in the far field, one could produce finite energy solutions (for instance, see \cite{GHJ23, Merle22b}). To illustrate the change of behavior of solutions near $\mathcal O=(0,0)$ in the original variables $t$ and $x$, we summarize the local behavior near $x=0$ from $t=-T$ to $t=T$ in the following.  

\begin{remark}\label{remark 2.3}
 The velocity of the solutions at initial time $t=-T$ near the vacuum boundary 
takes the from $$u(-T,x) = - \tfrac{2}{(\gamma+1)T} x + c_1 x^{\frac{1}{1-\mu}}+ o(x^{\frac{1}{1-\mu}}), \quad 0<x\ll 1 $$ where $c_1>0$ is an arbitrary constant (cf. Remark \ref{remark_initial}); when $t\to 0^\pm$ and $x>0$, it takes the form $$u(t,x) = - c_2 x^{1-\mu} + o(|t|), \quad x>0, \ |t|\ll 1 $$
 for some $c_2>0$ (cf. Remark \ref{remark3.8}); when $t>0$ and $|x| \ll t^\delta$, it behaves like $$u(t,x)=-c_3 t^{\delta -1} + o(t^{\delta -1}), \quad t>0,  \ |x| \ll t^\delta$$ for some $c_3>0$ (cf. \eqref{u at E}); 
 when $t>0$ and $ ( y_D-\epsilon) t^\delta <x < (y_D + \epsilon) t^\delta$ for some sufficiently small $\epsilon>0$, 
 \[
u(t,x) = 
\begin{cases}
 \delta \frac{x}{t} (\frac{2}{3-\gamma} + c_4 ( \frac{x}{t^\delta} - y_D) + o( \frac{x}{t^\delta} - y_D ) ), &y_D t^\delta <x < (y_D + \epsilon) t^\delta \\
 \delta \frac{2}{3-\gamma} \frac{x}{t},  & x=y_D t^\delta  \\
  \delta \frac{x}{t} (\frac{2}{3-\gamma} + c_5 ( \frac{x}{t^\delta} - y_D) + o( \frac{x}{t^\delta} - y_D ) ), & (y_D-\epsilon) t^\delta <x < y_D t^\delta 
\end{cases}
 \]
 for $c_4, c_5>0$, $c_4\neq c_5$ (cf. Lemma \ref{Le: asy at D both sides}); when $t>0$ and $ y_B t^\delta \le x < (y_B+ \epsilon) t^\delta$, $$u(t,x )= \delta \frac{x}{t} (1 + c_6 ( \frac{x}{t^\delta} - y_B) + o( \frac{x}{t^\delta} - y_B ) ), \quad h(t,x )=\frac{ \delta^2}{\gamma-1} \frac{x^2}{t^2} (c_7 ( \frac{x}{t^\delta} - y_B) + o( \frac{x}{t^\delta} - y_B ) ) $$ 
 for $c_6, c_7>0$ (cf. Lemma \ref{Asy at B}). We note that constants $c_j >0$, $2\le j\le 7$ depend on $c_1>0$ as well as $\gamma$ and $\mu$. 
\end{remark}


\section{Connection from $C$ to $A$, $E$ and $D$}\label{Sec 3}


\subsection{Basic phase portrait analysis}

There are five relevant critical points {(see Figure \ref{fig2})}. 

\

\underline{\textbf{Zeros of $F=G=\Delta=0$:}} 
\begin{align}
&B=(U_B,H_B)=(1,0), \\ 
&D=(U_D,H_D) = (\frac{2}{3-\ga},(\frac{\ga-1}{3-\ga})^2).
\end{align}
The triple point $B$ is the saddle for $y=y_B<0$, and the triple point $D$ is the node for $y=y_D<0$.

\

\underline{\textbf{Zeros of $F=G=0$ and $\Delta\neq 0$:}} 

\begin{align}
&A=(U_A,H_A)=(0,0), \\ 
&C=(U_C,H_C) =( \frac{2}{\ga+1}\mu,(\frac{\ga-1}{\ga+1})^2\mu^2), \\ 
&E=(U_E,H_E) = (\infty,\infty).
\end{align}
The point $A$ is the node relevant for $|y| \rightarrow \infty$, $C$ the saddle for $y \rightarrow 0$, $E$ the node-saddle for $ |y| \rightarrow 0$. There exists another double point $ (\mu,0)$, but this point is irrelevant to our analysis. 

\

The goal is to construct a smooth trajectory 
to the ODE system \eqref{ode system} which starts at the point $C$ ($x=0$ for $t<0$),  and continues to the points $A$ ($x>0$ for $t=0$), $E$ ($x=0$ for $t>0$), $D$ (the first sonic point: weak discontinuity for $t>0$) and ends at the triple point $B$ (the second sonic point: vacuum boundary for $t>0$). 

\begin{figure}
	\centering
	\includegraphics[scale=0.7]{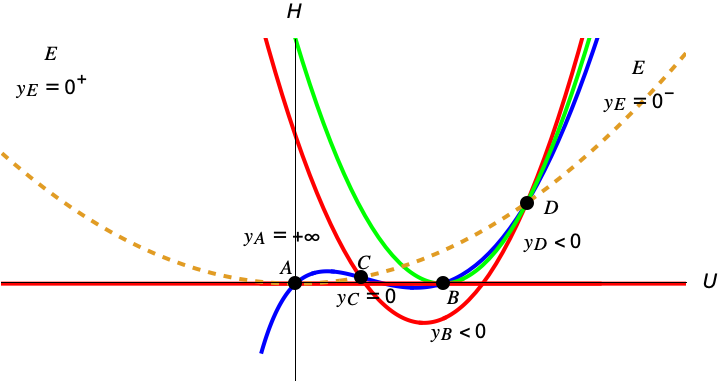}
	\caption{$\gamma = 1.816$, $\mu=0.716$, Dashed line: $H^\text{sp} = \frac{(\ga-1)^2}{4}U^2$, Red line: $F=0$, Blue line: $G=0$, Green line: $\Delta=0$.}
	\label{fig2}
\end{figure}

\begin{remark}\label{remark3.1}
It is clear that points $A$, $C$ and $D$ lie on the curve $H = \frac{(\ga-1)^2}{4}U^2$. 
\end{remark}

For all the critical points, we have $F = G =0$. Hence, by using the L'Hospital Rule, at each critical point, we obtain
\begin{align}
	H'(U) = \frac{F_H - G_U \pm \sqrt{(G_U-F_H)^2+4F_UG_H}}{2G_H}
\end{align}
where
\begin{align*}
	&F_H(U,H) = 4H - 2(U^2-k_1U+\mu),\quad F_U(U,H)  = -2H(2U-k_1),\\
	&G_H(U,H)  = U+k_2,\quad G_U(U,H) = H-3U^2+2(1+\mu)U-\mu.
\end{align*}
 \begin{remark}\label{remark3.2} In this paper, we are interested only in the case $\delta>1$, which is equivalent to $\mu\in(0,1)$. When $\delta=1$ (i.e., $\mu$=1), the self-similar variables $(U,H)$ either solve 
\begin{align}\label{delta1case1}
y\cfrac{dH}{dy} = -2H\quad \text{and}\quad y\cfrac{dU}{dy} =-U, 
\end{align}
or satisfy 
\be\label{delta1case2}
H = (U-1)^2. 
\ee
In the case of \eqref{delta1case1}, 
by directly solving the ODEs \eqref{delta1case1}, we obtain $H(y) = K_1 y^{-2}$ and  $U(y) = K_2 y^{-1}$
for some constants $K_1, K_2$. Returning to the physical variables, for $t<0$ and $x>0$, we have 
$
	h(t,x) \equiv K_1$ and  
	$u(t,x) \equiv -K_2,
$
which correspond to the standard Riemann problem, a scenario outside 
of our interest. 
In the case of \eqref{delta1case2}, 
we plug 
\eqref{2.14} and \eqref{delta1case2} back to the second equation of \eqref{Eulerh} to obtain
\be
	y(U-1)U'(y) = -(U-U_C)(U-1).
\ee
If $U-1=0$, then $H\equiv 0$ and this is not relevant to our study. For $U-1\neq 0$, the ODE reduces to
\be
	yU'(y) = -(U-U_C).
\ee
From the ODE, $U=U_C$ is a stationary solution, and all other solutions converge to $U_C$ as $|y|\to \infty$. Now, to satisfy $\rho_0(0)=0$, the only admissible solution is $U\equiv U_C$ and $H\equiv H_C$. Back to \eqref{2.14}, we see that the physical variables satisfy 
\be\label{delta12}
	u(t,x) = -U_C \frac{x}{t}\quad \text{and}\quad h(t,x) = \frac{H_C}{\ga-1} \frac{x^2}{t^2}.
\ee
This solution exhibits blow-up behavior for $t\to 0^-$ and $x>0$, but it 
is not locally integrable when $t\to 0^-$. 
We note that this profile in  \eqref{delta12}  aligns with the behavior described in \cite{Camassa20} near $x=0$ for $t<0$.

On the other hand, 
if $0<\delta<1$ (i.e. $\mu>1$), while the special solution $H^\text{sp}(U)$ still exists, the corresponding ODE solution (cf.  \eqref{formula special}) does not allow the trajectory to begin at the point $C$ ($y=0$). 
\end{remark}

\subsection{Connection from $C$ to $A$}

We first show that there are two possible slopes at $C$. 

\begin{lemma}\label{der at C}
	For any $\ga\in(1,3)$ and $\mu\in(0,1)$, there are two distinct branches at the point $C$. 
\end{lemma}
\begin{proof}
	Direct computation shows that there are two distinct branches: 
	\begin{align}
		C_1(U_C,H_C) &= \frac{(\ga-1)^2}{\ga+1}\mu>0, \\
		 C_2(U_C,H_C)& = \frac{(\ga-1)\Big[(\ga^2-2\ga+5)\mu-(\ga+1)^2\Big]\mu}{2(\ga+1)(\ga+1-2\mu)}<-\frac{4(\ga-1)^2\mu}{2(\ga+1)(\ga+1-2\mu)}<0.
	\end{align}
\end{proof}




To connect $C$ to $A$, we can only choose the branch associated with the positive slope $C_1$. On the other hand, the slope of the special solution $ H^{\text{sp}}(U)= \frac{(\gamma-1)^2}{4}U^2$ at the point $C$ is the same as $C_1$: 
\[
{H^{\text{sp}}}' (U_C)= \frac{(\gamma-1)^2}{2} U_C =  \frac{(\gamma-1)^2}{\gamma+1}\mu = C_1(U_C,H_C). 
\]
Since $H^{\text{sp}}(U_C) = H_C$ as in Remark \ref{remark3.1}, by the uniqueness of smooth solutions of \eqref{ODE}, the solution starting from $C$ with the slope $C_1$ must be $H^{\text{sp}}$, and hence $C$ must be connected to $A$ via the special solution $H^{\text{sp}}(U)= \frac{(\gamma-1)^2}{4}U^2$.

Before we move on, we examine the dynamics and properties of the special solution $H^{\text{sp}}(U) = \frac{(\ga-1)^2}{4}U^2$ satisfying our ODE system \eqref{ode system} in $y$ variable. Plugging  $H(U) = \frac{(\ga-1)^2}{4}U^2$ into the second equation of \eqref{ode system}, we obtain the equation for $U=U(y)$:
\be\label{ode B}
y \frac{dU}{dy} = -\frac{[(\ga+1)U-2\mu]U}{(\ga+1) U-2}. 
\ee
Since our trajectory starts at $C$, we must have $U(0)=U_C= \frac{2\mu}{\gamma+1}$. And since $y=0$ is a singular point, the solution to \eqref{ode B} cannot be determined uniquely by using $U(0)=\frac{2\mu}{\gamma+1}$ only. Local asymptotic analysis around $y=0$ reveals that \eqref{ode B} admits series solutions around $y=0$ in terms of powers of $y^{\frac{\mu}{1-\mu}}$, for which the coefficient of $y^{\frac{\mu}{1-\mu}}$  acts as a free parameter in the series expansion. In fact, \eqref{ode B} can be integrated in terms of $U$ to yield the explicit form of $y$ in terms of $U$. 

\begin{lemma}[Special solution before the change of behavior]\label{lemma 3.4}  Let $\mu\in(0,1)$ be given. The solutions of \eqref{ode B} with $U(0)= \frac{2\mu}{1+\gamma}$ satisfy the following 
\be\label{B formula}
y= K \frac{ ( \frac{2\mu}{\gamma+1} -U^{\text{sp}})^{\frac{1}{\mu}-1}}{ (U^{\text{sp}})^\frac{1}{\mu}}  \ \text{ for }  \ 0< y <\infty,  
\ee
and 
\be\label{B bc}
   U^{\text{sp}}(0)= U_C= \frac{2\mu}{\gamma+1} , \  \ \lim_{y\to\infty} U^{\text{sp}} (y)  =U_A=0, \text{ and }  0=U_A< U^{\text{sp}} < U_C,
\ee
where $K>0$ is an integral constant. Moreover,  $U^{\text{sp}}$ enjoys the following asymptotic behavior: as $y\to 0^+$ 
\be\label{Asy at C}
U^{\text{sp}} (y) = U_C - \frac{U_C^\frac{1}{1-\mu}}{K^{\frac{\mu}{1-\mu}}} y^{\frac{\mu}{1-\mu}} + o (y^{\frac{\mu}{1-\mu}} ),  \ \  0<y \ll 1
\ee
and as $y\to \infty$, 
\be\label{Asy at A}
U^{\text{sp}}(y) =  K^\mu U_C^{1-\mu}\frac{1}{y^\mu} + o(\frac{1}{y^\mu}), \ \ y \gg  1  
\ee
\end{lemma}

\begin{proof} Let $z=\ln y$ for $y>0$. Then \eqref{ode B} takes the form
\begin{align}
	\frac{dU}{dz} = -\frac{[(\ga+1)U-2\mu]U}{(\ga+1) U-2}
\end{align}
which can be written as 
\begin{align}
		&\cfrac{dz}{dU} = -\cfrac{(\ga+1) U-2}{[(\ga+1)U-2\mu]U} = -\cfrac{1}{\mu U}+(\frac{1}{\mu}-1)\cfrac{1}{U-\frac{2\mu}{\gamma+1}}. 
\end{align}
By integrating, we obtain 
\be
z= \alpha -\frac{1}{\mu}\ln |U| + (\frac{1}{\mu}-1)\ln | U- \frac{2\mu}{\gamma+1}|
\ee
where $ \alpha \text{ is an integral constant}$. Going back to $y= e^z$, we obtain the formula \eqref{B formula}, while   \eqref{B bc},  \eqref{Asy at C} and \eqref{Asy at A} readily follow from \eqref{B formula}. 
\end{proof}

\begin{remark}\label{remark_initial}
The integral constant $K>0$ in \eqref{ode B} can be related to the next order expansion of the initial data for the special solution when going back to original variables $(t,x)$. In fact, from \eqref{2.14} and \eqref{Asy at C}, we see that for $t=-T <0$ and sufficiently small $x>0$, 
\be\label{(3.22)}
u(-T,x) = \frac{\delta }{ -T } x U^{\text{sp}} (\frac{x}{T^\delta}) = - \frac{2 }{ (\gamma+1)T}  x +  \frac{U_C^\frac{1}{1-\mu}}{ \mu K^{\frac{\mu}{1-\mu}} T^{\frac{2-\mu}{1-\mu}}} x^{\frac{1}{1-\mu}} + o (x^\frac{1}{1-\mu})
\ee
where we have used $\delta= \frac{1}{\mu}$ to express the far-right had side in terms of $\mu$. The parameter $\mu\in (0,1)$ dictates the initial regularity of the profile: $u_0 \in C^{[\beta], \beta -[\beta]}$ if $\beta \neq \mathbb N$ and $u_0\in C^\infty$ if $\beta\in \mathbb N$, where $\beta = \frac{1}{1-\mu}\in (1,\infty)$.
\end{remark}

\subsection{Continuation through $A$, $E$ and $D$}

Once the trajectory of the ODE system \eqref{ode system} arrives at $A$, it may continue to the second quadrant in $(U,H)$ plane.   As moving from the first quadrant to the second quadrant, the extension may not be unique, in particular if one doesn't insist on the smoothness of the trajectory at the point $A$. When going from $A^-$ to $A^+$ ($t=0^-$ to $t=0^+$), the only reasonable extension for the physical velocity $u(t,x)$ (see \eqref{2.14} and \eqref{2.15})  
turns out to be via the special solution $ H^{\text{sp}}(U)= \frac{(\gamma-1)^2}{4}U^2$ for $U<0$, which is indeed the smoothest passage, and then follows the special solution through $E$ and subsequently $D$. In what follows, we discuss this continuation through $A$, $E$ and $D$ in detail.  

We first consider the extension of the solution via the special solution $H^{\text{sp}}(U)= \frac{(\gamma-1)^2}{4}U^2$ from $A^-$ to $A^+$ ($t=0^-$ to $t=0^+$). 
The same integration argument  in the proof of Lemma \ref{lemma 3.4} using $z= \ln |y|$ gives rise to the formula for the solutions of \eqref{ode B}
\be\label{formula special}
|y| = K \frac{ |U^\text{sp}-\frac{2\mu}{\gamma+1} |^{\frac{1}{\mu}-1}}{|U^\text{sp}|^\frac{1}{\mu}}
\ee
for $y$ irregardless of $y= \frac{x}{(-t)^\delta}$ or $y= \frac{x}{t^\delta}$. This motivates us to extend the solution $U^{\text{sp}}$ obtained in Lemma \ref{lemma 3.4} satisfying $0<U^{\text{sp}}<U_C$ for $0<y= \frac{x}{(-t)^\delta}<\infty$ and $t<0$,  to $U^{\text{sp}}<0$ in the region $0 <y= \frac{x}{t^\delta} <\infty$ and $t>0$ smoothly via the formula \eqref{formula special}. To ensure a smooth extension, we demand 
\be\label{B Bc at A}
\lim_{y\to \infty, t>0} U^{\text{sp}}(y) =\lim_{y\to \infty, t < 0} U^{\text{sp}}(y)=0 \ (=U_A)
\ee
 where the first $U$ is from \eqref{2.15} and the second $U$ is from \eqref{2.14}. 

 \begin{lemma}[Extension of $U^{\text{sp}}$ from $A$ to $E$] \label{lemma 3.6} Define $U^{\text{sp}}=U^{\text{sp}}(y)$ for $0<y<\infty$ by the following implicit formula
 \be\label{formula AE}
y= K \frac{ (\frac{2\mu}{\gamma+1} - U^{\text{sp}} )^{\frac{1}{\mu}-1}}{( - U^{\text{sp}})^\frac{1}{\mu}} \text{ for }0<y<\infty. 
\ee
Then $U^{\text{sp}}$ solves \eqref{ode B} and $-\infty = U_E< U^{\text{sp}} < U_A=0 $ with \eqref{B Bc at A}. Moreover, it holds that near  the point $A$ (as $y\to \infty$)
\be\label{Asy at E}
U^{\text{sp}}(y) = - K^\mu U_C^{1-\mu}\frac{1}{y^\mu} + o(\frac{1}{y^\mu}), \ \ y\to \infty 
\ee
and near the point $E$ (as $y \to 0^+$) 
\be\label{Uy0+}
U^{\text{sp}}(y)  = - \frac{K}{y} + o(\frac{1}{y}), \ \ y \to 0^+ . 
\ee
\end{lemma}

\begin{proof} All the claims follow directly from the formula \eqref{formula AE}, which solves \eqref{ode B}. 
\end{proof}

Next we show that the special solution $H^{\text{sp}}(U)$ is the unique choice of the extension in the region $\{(t,x): t\in (-T,T), \ x >0\}$, as a self-similar weak solution to the Euler system \eqref{Euler}. 

\begin{lemma}\label{PassA} Let $\gamma\in(1,3)$ and $\mu\in(0,1)$ be given. Suppose that the solution $(U,H)$ to the self-similar ODE system \eqref{ode system} where $(U,H)$ is induced by the special solution for $t<0$ yields a weak solution $(\rho, u)$ via \eqref{2.14} and \eqref{2.15} to the Euler system for all $t\in (-T,T)$. 
Then the trajectory in the first quadrant for $U>0$ and $t<0$ must pass through the point $A=(0,0)$  and continue to the second quadrant for $U<0$ and $t>0$ by following the special solution $H^{\text{sp}}(U)=  \frac{(\gamma-1)^2}{4}U^2$. 
\end{lemma}
\begin{proof}
In order to prove that only the special solution $H^{\text{sp}}(U)= \frac{(\gamma-1)^2}{4}U^2$ gives rise to a weak solution to the Euler  system, we need to show that any other choice of continuation of solutions to the ODE to the left side of $A$ (in the second quadrant) does not lead to a weak solution to the original system, whereas the special solution does. We first show that other choices do not yield weak solutions.

By the classical local analysis of the ODE system around $A$, where $A$ corresponds to $y=+\infty$, the solution should have the asymptotic behavior:
\begin{align}
 	U(y) &= U_1 y^{-\mu} + o(y^{-\mu}),\\
	H(y) &= H_1 y^{-2\mu} + o(y^{-2\mu}),
\end{align}
where $U_1H_1 \neq 0$. For any non-special solutions, $\frac{U_1^2}{H_1} \neq \frac{(\gamma-1)^2}{4}$. Thus, at least one of $U_1$ or $H_1$ differs from the special solution. Without loss of generality, we assume that $H_1$ is different from that of the special solution $H_s= \frac{(\gamma-1)^2}{4} K^{2\mu} U_C^{2(1-\mu)}$. Then, plugging this profile back into the ansatz \eqref{2.14} and \eqref{2.15}, we obtain that for any $0<x<1$,
\begin{align}
 	h(t,x) \to \frac{H_s}{\ga-1} x^{2-2\mu} + o(x^{2-2\mu}), \ \text{as} \ t \to 0^-,\\
	h(t,x) \to \frac{H_1}{\ga-1} x^{2-2\mu} + o(x^{2-2\mu}), \ \text{as} \ t \to 0^+.
\end{align}
Since $H_1 \neq H_s$, we deduce that for any fixed $0<x_0<1$, $h(t,x_0)$ experiences a jump across $t=0$, while $u$ and $h$ are smooth in each region $\{(t,x): t<0, x>0\}$ and $\{(t,x): t>0, x>0\}$. Now for any fixed $0<x_0<1$, consider a test function $\phi \in C_c^1 ((-\epsilon, \epsilon) \times (x_0-\sigma,x_0+\sigma))$ for sufficiently small $\epsilon >0$ and $\sigma>0$ satisfying $\phi (0,x_0) \neq 0$ and $\int_{x_0-\sigma}^{x_0+\sigma} x^{\frac{2-2\mu}{\gamma-1}} \phi (0, x)  dx \neq 0$. Then 
\begin{align}
	&\int_{-T}^T\int_{\Omega(t)} \rho \phi_t+\rho u  \phi_x d x d t  = \int_{-\epsilon}^\epsilon \int_{x_0-\sigma}^{x_0+\sigma} \rho \phi_t+\rho u  \phi_x d x d t \notag\\  
	&= \int_{-\epsilon}^\epsilon \int_{x_0-\sigma}^{x_0+\sigma} \Big(\rho_t +(\rho u)_x\Big) \phi \, dx \, dt + \int_{x_0-\sigma}^{x_0+\sigma}(\rho \phi)|_{t=0^+} - (\rho \phi)|_{t=0^-} \, dx\label{weak sol} \\
	&= \int_{x_0-\sigma}^{x_0+\sigma} \Big(\frac{H_1^\frac{1}{\gamma-1}-H_s^\frac{1}{\gamma-1}}{(\ga-1)^{\frac{1}{\gamma-1}}} x^{\frac{2-2\mu}{\gamma-1}} + o( x^{\frac{2-2\mu}{\gamma-1}} )\Big) \phi \, dx \neq 0, \text{ if } H_1\neq H_s\notag
\end{align} 
which shows that the corresponding extension with $H_1\neq H_s$ does not yield a weak solution. Of course, it is clear that the extension by the special solution guarantees the continuity for all $x>0$ across $t=0$ as well as smoothness for each region $t<0$ and $t>0$. Therefore, the conclusion follows. 
\end{proof}

We now discuss the detailed behavior of the solution that we have constructed so far in terms of the physical velocity $u(t,x)$ near $(0,0)$, the first point of the singularity for $x>0$.

\begin{remark}\label{remark3.8}
Substituting \eqref{Asy at C} into the ansatz \eqref{2.14}, we obtain the following behavior for the physical velocity $u$ when $x = o((-t)^\delta)$ and $t \to 0^-$ (in the sense of $y \sim 0$) 
$u(t,x) = \frac{2}{\ga+1}\frac{ x}{t} + o(x)$. Hence, the initial behavior given in \eqref{(3.22)} persists in a shrinking region $x = o((-t)^\delta)$ as $t \to 0^-$, and it will undergo the change of behavior. 
In particular,  in the regime where $x>0$ is fixed, $t\to 0^\pm$ and $y = \frac{x}{|t|^\delta}\to+\infty$, from \eqref{Asy at E}, the physical velocity $u$ behaves like 
\be\label{u at A}
	u(t,x) =- \delta K^\mu U_C^{1-\mu} x^{1-\mu} + o(|t|), \quad x>0. 
\ee
\end{remark}

The behavior \eqref{u at A} is valid for any fixed $x>0$ and sufficiently small $|t|$ ($x \gg |t|^\delta$). On the other hand, the solutions constructed so far do not reveal much in the region $x \le  |t|^\delta$ including $x=0$.  

We can further extend the solution obtained in Lemma \ref{lemma 3.6} for $0<y<\infty$ and $t>0$ to the region $y\le 0$ and $t>0$ by requiring the continuity of $yU(y)$ at the point $E$ from $y=0^+$ to $y=0^-$ 
\be\label{condition}
\lim_{y\to 0^-} y U^{\text{sp}}(y) =  \lim_{y\to 0^+} y U^{\text{sp}}(y) =-  K 
\ee
so that $U>0$ for $y<0$ until the special solution reaches the sonic point, namely $U=U_D = \frac{2}{3-\gamma} >0$. To this end, we first define $y_D$ as 
\be\label{yD}
y_D = - K \frac{ (U_D - \frac{2\mu}{\gamma+1})^{\frac{1}{\mu}-1}}{ U_D^\frac{1}{\mu}} <0. 
\ee  

\begin{lemma}[Extension of $U^{\text{sp}}$ from $E$ to $D$] \label{lemma 3.7}
Define $U^{\text{sp}}=U^{\text{sp}}(y)$ for $y_D< y<0$ by the following 
\be
y= - K \frac{ (U^{\text{sp}} - \frac{2\mu}{\gamma+1})^{\frac{1}{\mu}-1}}{ (U^{\text{sp}})^\frac{1}{\mu}} \ \ \text{for} \ \   y_D< y<0. 
\ee
Then $U^{\text{sp}}$ solves \eqref{ode B} with \eqref{condition}, and $\frac{2}{3-\gamma}=U_D < U^{\text{sp}} <U_E=\infty $. 
 \end{lemma}

\begin{proof} The claims directly follow from the formula. 
\end{proof}

We observe that once the solution is extended through $E$, the solution trajectory in $(U,H)$ plane is moved from the second quadrant to the first quadrant such that $U>0$ and $H>0$. 

The continuity of the extension $yU^{\text{sp}}$ at $y=0$ (at the point $E$) (see \eqref{condition}) is deciphered into the physical velocity as 
\be\label{u at E}
u(t,x) = \delta t^{\delta-1} y U(y) = -K \delta  t^{\delta-1}+o(t^{\delta-1}),  \text{ for }   |x| \leq \epsilon t^\delta  \ \text{and} \ t>0 
\ee
for all sufficiently small $\epsilon>0$. This in particular ensures H\"{o}lder continuity in time near $x=0$, and it also shows a drastic change of behavior of $u$ from  \eqref{u at A} in the region $x\gg t^\delta$ to \eqref{u at E} in the region $|x| \ll t ^\delta$.  

It is evident that any other choice of the leading order coefficient of $y U $  for $y\sim 0^-$ that differs from $-K$ will result in a discontinuity at $x = 0$ for any time $t>0$. As discontinuities are allowed for the compressible Euler flows, one might explore the possibility of extensions at $x=0$ via shock discontinuities.  
 However, as noted by Jenssen \cite{Jenssen} this type of solution will lead to only a partially defined solution, which we discuss in the following remark.  
 
\begin{remark}\label{E smooth extension}
First of all, the shock could only be a 2-shock since the shock speed is 0, and the solution with the shock profile satisfies the Rankine–Hugoniot conditions and entropy condition:
\begin{align}
	\rho_-u_- &= \rho_+u_+,\label{E1}\\
	\rho_-u_-^2+p_- & = \rho_+u_+^2+p_+,\label{E2}\\
	u_-+\sqrt{p'(\rho_-)} > &0 > u_++\sqrt{p'(\rho_+)} ,\label{E3}
\end{align}
where the subscript "-" and "+" always refer to the evaluation at the left and right of $x=0$ in physical space, respectively. 
Moreover, from \eqref{2.15}, we have $u_-<0$. Then, together with \eqref{pressure} and \eqref{E3} we have
\be
	u_-+\sqrt{p'(\rho_-)} = u_-+\sqrt{(\ga-1)h_-}>0 \Longrightarrow h_-> \frac{1}{\ga-1} u_-^2.
\ee 
Returning to the self-similar ansatz \eqref{2.15}, the state $(u_-, h_-)$ with $ h_- > \frac{1}{\gamma - 1} u_-^2$ corresponds to a point $(U_-, H_-)= (+\infty, +\infty)$ where 
$H_- : = \lim_{y\to 0^-}H(y)=+\infty$, $U_- : = \lim_{y\to 0^-}U(y)=+\infty$ and $\frac{H_-}{U_-^2} > 1$. From the equation $F=2H[H-(U^2-k_1U+\mu)]$, it is clear that $(U,H)$ ($y<0$) continued from $(U_-,H_-)$ lies above the curve ${F = 0}$ in the first quadrant of $(U,H)$ plane. In this case, both $ U'(y) > 0$ and $ H'(y) > 0$, and the flow descends as $y$ decreases. 
Based on the phase portrait analysis, there are four possible scenarios, as depicted in Figure \ref{fig2}:
\begin{enumerate} 
\item The solution reaches the curve $ H(U) = \frac{(\gamma - 1)^2}{4} U^2$ (special solution) at a point other than $ D$. 
\vspace{-1.5mm}
\item The solution reaches the curve ${\Delta = 0}$ (sonic curve) at a point other than  $ D$. 
\vspace{-1.5mm}
\item The solution reaches the curve ${F = 0}$ before reaching $ D$. 
\vspace{-1.5mm}
\item The solution reaches $ D$. 
\end{enumerate}
Case 1 is impossible due to the uniqueness of the solution to the ODE system \eqref{ODE}  at a regular point.
 If Case 2 occurs, both $ H'(y)$ and $ U'(y)$ tend to infinity, and the trajectory can't continue uniquely, which is not the solution we seek.  
For Case 3, we proceed by contradiction. Let $ H = H_-(U)$ represent the solution continued from $ (U_-, H_-)$, and denote $ H = H_F(U)$ as the curve ${F = 0}$. Suppose the solution curve first intersects $ H = H_F(U)$ at $ U_0 \in (U_D, \infty)$. Then, $ H_-(U) > H_F(U)$ for all $ U \in (U_0, \infty)$, and $ H_-'(U_0) = 0$. However, $ H_F'(U_0) > 0 = H_-'(U_0)$, which contradicts the assumption that $ H_-(U) > H_F(U)$ for all $ U \in (U_0, \infty)$. Therefore, Case 3 cannot occur. 
In Case 4, we have that $ H_F'(U_D) = 2U_D - k_1 = \frac{\gamma^2 - 2\gamma + 5 - \mu(3 - \gamma)^2}{2(3 - \gamma)} > C_2(U_D, H_D)$, where $ C_2(U_D, H_D)$ is defined in \eqref{C2D}. From Lemma \ref{Le: integral curves at D}, it follows that all other solution curves, aside from the special solution curve, converge to $ D$ with slope $ C_2(U_D, H_D)$ at $ D$. Consequently, the solution curve from $ (U_-, H_-)$ cannot reach $ D$.
\end{remark}

\section{Connection from $B$ to $D$}\label{Sec 4}

What remains is to show the existence of a trajectory  connecting $D$ and $B$. 
We first analyze the point $D$.
\begin{lemma}\label{der at D}
	For any $\ga\in(1,3)$ and $\mu\in(0,1)$, there are two distinct branches at point $D$. 
\end{lemma}
\begin{proof}
	Direct computation shows that there are two distinct branches (slopes)
	\begin{align}
		C_1(U_D,H_D) &= \frac{(\ga-1)^2}{3-\ga}>0, \\ 
		 C_2(U_D,H_D) &= \frac{(\ga-1)\Big[-(\ga-3)^2\mu+\ga^2-2\ga+5\Big]}{2(3-\ga)[(\ga-3)\mu+2]}>\frac{(\ga-1)^2}{2(3-\ga)[(\ga-3)\mu+2]}>0.\label{C2D}
	\end{align}
	Moreover, we check
	\begin{align*}
		-(\ga-3)^2\mu+\ga^2-2\ga+5 - 2(\ga-1)[(\ga-3)\mu+2] = (3-\ga)[(3\ga-5)\mu+3-\ga]>0.
	\end{align*}
	Hence, $C_2(U_D,H_D)>C_1(U_D,H_D)>0$.
\end{proof}

The branch associated with the slope $C_1(U_D,H_D)$ is the one corresponding to the special solution $ H^{\text{sp}}(U)= \frac{(\gamma-1)^2}{4}U^2$. If the trajectory continues along this branch from $D$ to the left, it will return to $C$ and to $A$ in the phase-portrait $(U,H)$ plane. 
 The trajectory approaches $A$ again with $y\to -\infty$. In particular, the resulting solution $(u(t,x),h(t,x))$ would be  defined for every point $(t,x)$ with $t>0$ 
 and $x<0$.  In that case, for any fixed $x<0$, $t\to 0^+$ corresponds to $y\to -\infty$, and hence by using \eqref{2.15} and \eqref{Asy at E}, we have
\begin{align}
	h(t,x) = \delta^2 \tau^{2(\delta-1)} y^2 \frac{H(y)}{\gamma-1} \sim \tilde K |x|^{2(1-\mu)}\quad \text{as }t\to 0^+
\end{align}
where $\tilde K$ is a nonzero and positive constant. However, this doesn't give a meaningful solution to our original problem \eqref{Euler} as $t=0^+$ is not a proper boundary, and moreover, it doesn't yield the vacuum boundary. 
 In order to connect $B$, it is then clear that the trajectory between $B$ and $D$ must connect $D$ with the other slope  $C_2(U_D,H_D)$ as it approaches $D$ from the left. Therefore, the solution cannot be $C^1$  at point $D$, but only Lipschitz continuous. 


Now, we analyze the regularity for the solution curves of \eqref{ode system} that pass through $D$.
\begin{lemma}\label{Le: integral curves at D}
	Let $\ga\in(1,3)$ and $\mu\in(0,1)$. There is only one integral curve that passes through $D$ with $\frac{dH(U)}{dU}\mid_{U=U_D}=C_1$. Specifically, the integral curve is the special solution $H(U)=\frac{(\ga-1)^2}{4}U^2$. All the other integral curves, passing through $D$, have $\frac{dH(U)}{dU}\mid_{U=U_D}=C_2$.
\end{lemma}

\begin{proof}
In order to further analyze the solution around $D$, it is helpful to consider the dynamical system under the change of variables $y \to s$ where $y\Delta(U,H) \frac{d}{dy} =  \frac{d}{ds} $. Hence, $\frac{ds}{dy} = \frac{1}{y[(U(y)-1)^2+H(y)]}$. We shall integrate it near $y=y_D$ with $\frac{dH(U)}{dU}|_{U=U_D}=C_i (U_D,H_D)$ where $i=1,2$. 
By using $\Delta(U(y_D),H(y_D))=0$, we expand $\Delta(U(y),H(y))$ around $y=y_D$ and obtain 
\begin{align}
\Delta(U(y),H(y)) = [2(U_D-1)-C_i(U_D,H_D)]U'(y_D)(y-y_D)+o (y-y_D).
\end{align}
It is a routine matter to check that $2(U_D-1)-C_i(U_D,H_D)>0$ for $i=1,2$. Moreover, since $(U_D,H_D)$ is a removable singularity of $\frac{G(U,H)}{\Delta(U,H)}$ and $|y_D|\in(0,\infty)$, it implies that $U'(y_D)\in(0,\infty)$. Consequently, 
\begin{align}\label{ODE: ds/dy}
	\frac{ds}{dy} = \frac{1}{y[(U(y)-1)^2+H(y)]} = \frac{1}{M(y-y_D)+o(y-y_D)}
\end{align}
where 
\be\label{def: M} 
M = [2(U_D-1)-C_i(U_D,H_D)]U'(y_D)y_D<0. 
\ee
It follows that 
\begin{align}\label{s(y)}
	s(y) = c_0 + \frac{1}{M}\ln|y-y_D| + o(y-y_D), \ \ y\to y_D, 
\end{align}
where $c_0$ is a constant. 
Therefore, $s\to +\infty$ as $y\to y_D$. Now, going back to the original ODE system \eqref{ode system}, after changing variables, we obtain
\begin{align}\label{ODE of s}
	\frac{dH}{ds}   = F(U,H)\quad \text{ and }\quad \frac{dU}{ds} = G(U,H).
\end{align}
We are interested in the behavior of the solution to the ODE system \eqref{ODE of s} around the equilibrium point $(U_D,H_D)$ as $s\to +\infty$. We first linearize the above system \eqref{ODE of s} at $D$, and denote
\begin{align}
	\vec X = \begin{pmatrix} H-H_D\\ U-U_D\end{pmatrix} 
\end{align}
Then, the ODE of $\vec X$ is given by
\begin{align}\label{eq: matrix form D}
		\cfrac{d \vec X}{ds} = \begin{pmatrix} F_H(U_D,H_D)&F_U(U_D,H_D)\\G_H(U_D,H_D)&G_U(U_D,H_D)\end{pmatrix}\vec X + \mathcal O(||\vec X||)=:\mathcal M \vec X+ \mathcal O(||\vec X||).
\end{align}
A direct computation shows that
\begin{align}
	\mathcal M=  \begin{pmatrix} \frac{2(\ga-1)^2}{(3-\ga)^2}&\frac{-(\ga-1)^2[\ga^2-2\ga+5-(3-\ga)^2\mu]}{(3-\ga)^3}\\\frac{4-2(3-\ga)\mu}{(3-\ga)(\ga-1)}&\frac{(\ga+1)(3-\ga)\mu+\ga^2-6\ga+1}{(3-\ga)^2}\end{pmatrix},
\end{align}
the eigenvalues of $\mathcal M$ are
\begin{align}
	\lambda_1 = \frac{2(\ga-1)(\mu-1)}{3-\ga}<0\quad \text{ and }\quad \lambda_2= \frac{(3-\ga)\mu-\ga-1}{3-\ga}<0,
\end{align}
and the corresponding eigenvectors are 
\begin{align}\label{eq: eigenvector D}
	\vec v_1 = \begin{pmatrix}C_2(U_D,H_D)\\ 1\end{pmatrix}\quad \text{ and }\quad
	\vec v_2 = \begin{pmatrix}C_1(U_D,H_D)\\ 1\end{pmatrix}.
\end{align}
It is a routine matter to check that 
\begin{align}
	\lambda_2< \lambda_1<0
\end{align}
for any $\ga\in(1,3)$ and $\mu\in(0,1)$. 
From the Poincare's Theorem and Poincar\'{e}-Dulac Theorem in \cite{Arnold12}, there exits a change of variables:
\begin{align}\label{map X to Y}
	\vec X  = \vec Y + f(\vec Y)
\end{align}
where $f$ is a vector-valued polynomial in $\vec Y = (y_1,y_2)^T$ of order $\geq 2$, and $f(0) = f'(0)=0$, such that the ODE system \eqref{eq: matrix form D} is transferred into
\begin{align}\label{ode of Y}
		\cfrac{d\vec Y}{ds} =  \begin{pmatrix} \lambda_1&0\\0&\lambda_2\end{pmatrix} \vec Y + a \begin{pmatrix} 0 \\  y_1^\frac{\lambda_2}{\lambda_1}\end{pmatrix}
\end{align}
where 
\begin{align*}
	a = \begin{cases} \text{nonzero constant},\ &\text{if} \ \frac{\lambda_2}{\lambda_1} \in \mathbb N,\\
	0, \ &\text{otherwise}. \end{cases}
\end{align*}
When $\frac{\lambda_2}{\lambda_1} \notin \mathbb{N}$, the ODE system \eqref{ode of Y} is a linear system that can be solved directly, and the solution can be mapped from $Y$ back to $X$ using the transformation from $X$ to $Y$. In the case where $\frac{\lambda_2}{\lambda_1} \in \mathbb{N}$, we first solve for $y_1$, and then solve the nonhomogeneous ODE for $y_2$ using the method of integrating factors. 
Hence, all the integral curves converging to $D$ satisfy for $\mathcal H= H-H_D$ and $\mathcal U= U-U_D$
\begin{align}\label{H(s), U(s) at D}
	 \begin{pmatrix}\mathcal H\\\mathcal U\end{pmatrix} = z_1 
	 \vec v_1 e^{\la_1 s} +z_2 \vec v_2 
	 e^{\la_2 s} + \mathcal O(\sum_{\substack{n+m = 2\\ n,m\geq 0}}z_1^nz_2^me^{(n\la_1+m\la_2)s}) \  \text{ for } \begin{pmatrix}\mathcal H\\\mathcal U\end{pmatrix}  \to 0 \text{ as }s\to +\infty 
\end{align}
where $z_1$ and $z_2$ are arbitrary constants with at least one of $z_1$ and $z_2$ being non-zero. It follows that
\begin{align}\label{eq: H'(U) local D}
	\frac{d\mathcal H}{d\mathcal U} \rightarrow 
	 \frac{z_1\la_1C_2e^{\la_1 s}+z_2\la_2C_1e^{\la_2 s}}{z_1\la_1e^{\la_1 s}+z_2\la_2e^{\la_2 s}} = \frac{z_1\la_1C_2+z_2\la_2C_1e^{(\la_2-\la_1)s}}{z_1\la_1+z_2\la_2e^{(\la_2-\la_1)s}} \ \ \text{as }s\to +\infty
\end{align} 
where we recall $C_1, C_2$ in connection to $\vec v_1$ and $\vec v_2$ from \eqref{eq: eigenvector D}. 
We consider three cases:

\underline{Case I: $z_1=0$ and $z_2\neq 0$.} By using \eqref{H(s), U(s) at D},  $\cfrac{\mathcal H(s)}{\mathcal U(s)} \to C_1$ as $s\to +\infty$. Consequently, the integral curve of \eqref{ode system} is along $\vec v_2$.

\underline{Case II: $z_2=0$ and $z_1\neq 0$.} By using \eqref{H(s), U(s) at D}, $\cfrac{\mathcal H(s)}{\mathcal U(s)} \to C_2$ as $s\to +\infty$. Consequently, the integral curve of \eqref{ode system} is along $\vec v_1$.

\underline{Case III: $z_1z_2\neq 0$.} \eqref{eq: H'(U) local D} implies $\cfrac{d\mathcal H}{d\mathcal U} \to C_2$ as $s\to +\infty$. Hence, the integral curves of \eqref{ode system} are along $\vec v_1$.

It then follows that there is only one integral curve of \eqref{ode system} along $\vec v_2$ which is $H(U) = \frac{(\ga-1)^2}{4}U^2$, and all the other integral curves are along $\vec v_1$. 
\end{proof}

For integral curves along $\vec v_1$, we derive the following behavior of $U$ and $H$ with respect to $y$. 

\begin{lemma}\label{Le: asy at D}
	Let $\ga\in(1,3)$ and $\mu\in(0,1)$. The integral curves, which pass through $D$ along $\vec v_1$, satisfy
	\begin{align}\label{Asy: D}
	U(y)&=U_D   -\frac{8+4(\ga-3)\mu}{(3-\ga)(\ga+1)y_D} (y-y_D) + U_\beta (y-y_D)^\beta + \mathcal O(\sum_{\substack{n+m = 2\\ n,m\geq 0}} U_\beta^m (y-y_D)^{n+m\beta}),\\
	H(y)&=H_D +  \frac{2(\ga-1)[(3-\ga)^2\mu-(\ga^2-2\ga+5)]}{(3-\ga)^2(\ga+1)y_D} (y-y_D) + H_\beta (y-y_D)^\beta + \mathcal O(\sum_{\substack{n+m = 2\\ n,m\geq 0}} H_\beta^m (y-y_D)^{n+m\beta})\notag 
\end{align}
where $\beta = \frac{\la_2}{\la_1}>1$, and $H_\beta$ and $U_\beta$ are arbitrary constants.
\end{lemma}

\begin{proof}
For any integral curves, which pass through $D$ along $\vec v_1$, satisfy $H'(U_D)=C_2(U_D,H_D)$. By using this identity and $y_D$ is finite, we have
\begin{align*}
	\lim_{y\to y_D^-} U'(y) &= \lim_{y\to y_D^-} \frac{1}{y_D}\frac{G_H(U,H)H'(U)U'(y)+G_U(U,H)U'(y)}{2(U-1)U'(y)-H'(U)U'(y)}\\
	& = \lim_{y\to y_D^-}  \frac{1}{y_D}\frac{G_H(U,H)H'(U)+G_U(U,H)}{2(U-1)-H'(U)}\\
	&=\frac{1}{y_D}\frac{G_H(U_D,H_D)C_2(U_D,H_D)+G_U(U_D,H_D)}{2(U_D-1)-C_2(U_D,H_D)}\\
	&=-\frac{8+4(\ga-3)\mu}{(3-\ga)(\ga+1)y_D}.
\end{align*}
This leads to 
\begin{align}
	U'(y_D) &= -\frac{8+4(\ga-3)\mu}{(3-\ga)(\ga+1)y_D},\label{U'(yD)}\\
	H'(y_D)&=  \frac{2(\ga-1)[(3-\ga)^2\mu-(\ga^2-2\ga+5)]}{(3-\ga)^2(\ga+1)y_D}.\label{H'(yD)}
\end{align}
Recalling \eqref{def: M}, \eqref{s(y)} and \eqref{H(s), U(s) at D}, we then obtain
\begin{align*}
	\la_1 s &= \la_1 c_0 + \frac{\la_1}{M}\ln|y-y_D| + o(y-y_D) = \la_1 c_0 + \ln|y-y_D| + o(y-y_D), \\
	\la_2 s &= \la_2 c_0 + \frac{\la_2}{M}\ln|y-y_D| + o(y-y_D) =\la_2 c_0+\beta \ln|y-y_D| + o(y-y_D), 
\end{align*}
where we have used $\la_1= M$, which follows from direct computation. Hence, 
\begin{align}
	U(y)&=U_D + K_1 (y-y_D) + U_\beta (y-y_D)^\beta + \mathcal O(\sum_{\substack{n+m = 2\\ n,m\geq 0}} K_1^nU_\beta^m (y-y_D)^{n+m\beta}),\label{U(y) at D}\\
	H(y)&=H_D + K_2 (y-y_D) + H_\beta (y-y_D)^\beta + \mathcal O(\sum_{\substack{n+m = 2\\ n,m\geq 0}} K_2^nH_\beta^m (y-y_D)^{n+m\beta}), \label{H(y) at D}
\end{align}
where $K_1$ and $K_2$ are constants 
to be determined. 
As observed, the remainders in both \eqref{U(y) at D} and \eqref{H(y) at D} have the lowest power of 2. Consequently, $K_1$ and $K_2$ must conform to \eqref{U'(yD)} and \eqref{H'(yD)} respectively. It completes the proof. 
\end{proof}
Combining Lemmas \ref{der at D}--\ref{Le: asy at D}, we conclude that
\begin{lemma}\label{Le: asy at D both sides}
	Let $\ga\in(1,3)$ and $\mu\in(0,1)$. The solution curve passing through $D$ satisfies
\begin{align*}
	&U(y) = U_D  -\frac{\ga+1-(3-\ga)\mu}{(\ga-1)(3-\ga)y_D}(y-y_D)+\sum_{n\geq 2}U_n (y-y_D)^n,  \text{ for }y-y_D>0,\\
	&U(y) = U_D-\frac{8+4(\ga-3)\mu}{(3-\ga)(\ga+1)y_D} (y-y_D) + U_\beta (y-y_D)^\beta + \mathcal O(\sum_{\substack{n+m = 2\\ n,m\geq 0}} U_\beta^m (y-y_D)^{n+m\beta}),  \text{ for }y-y_D<0, 
\end{align*}
where $\beta = \frac{\la_2}{\la_1}>1$, and $U_\beta$ is to be determined by Lemma \ref{B to D}.
\end{lemma}



We now present local analysis at the point $B$ corresponding to the vacuum boundary.  

\begin{lemma}\label{der at B}
	For any $\gamma \in (1,3)$ and $\mu \in (0,1)$, there exist two distinct branches at the point $B$: $C_1(U_B, H_B) = 0$ and $C_2(U_B, H_B) = \frac{\gamma(1-\mu)}{1+k_2}$. Moreover, there exists exactly one integral curve passing through $B$ along each of these branches, and both integral curves are smooth at $B$, and 
	the integral curve corresponding to $C_1(U_B, H_B) = 0$ is precisely $H\equiv 0$.  
\end{lemma}
\begin{proof}
	As in Lemma \ref{Le: integral curves at D}, we linearize the system around $B$. The eigenvalues of the matrix corresponding to the linearized system are	\[
	\lambda_1 = \mu-1<0<\lambda_2 = (\ga-1)(1-\mu).
	\]
	Therefore, by the classical theory of local analysis for ODE systems, $B$ is a saddle point. Consequently, there exists exactly one integral curve passing through $B$ in each direction.  The claim regarding the curve corresponding to $C_1(U_B, H_B) = 0$ follows directly from the uniqueness of the solution.  
\end{proof}

We use $U_G(H)$ or $H_G(U)$ to denote the curve described by $G=0$. 
\begin{lemma}\label{U_G'(H)>0}
Let $\ga\in(1,3)$ and $\mu\in(0,1)$. For $U\geq 1$,  $H_G(U)$ satisfies $0<H_G'(U)<\infty$.
\end{lemma}
\begin{proof}
	Since $H_G(U) = \frac{U(U-1)(U-\mu)}{U+k_2}$, recalling \eqref{k2},
	\begin{align*}
		H_G'(U) = \frac{(2U-1-\mu)U^2+k_2(3U-2(1+\mu)U+\mu k_2)}{(U+k_2)^2}\in(0,\infty)
	\end{align*}
	for any $U\geq 1$.
\end{proof}
We now prove that the integral curve from $B$ with $\frac{d H(U)}{dU} \big|_{U=U_B} = \frac{\gamma(1-\mu)}{1+k_2}$ leads to the  connection to $D$.
\begin{lemma}\label{B to D}
For any $\ga\in(1,3)$ and $\mu\in(0,1)$, there exists a unique integral curve, which connects $B$ to $D$ with the slope $C_2(U_B,H_B)>0$  at the point  $B$. Moreover, the asymptotic behavior of the integral curve satisfies \eqref{Asy: D} at the point $D$.
\end{lemma}
\begin{proof}
The uniqueness follows from Lemma \ref{der at B}, and the second part directly follows from the first part. Hence, we only need to show there exists an integral curve which connects from $B$ to $D$. The special solution $H^{\text{sp}}(U) = \frac{(\ga-1)^2}{4}U^2$ is a natural upper barrier. We claim $H_G(U)$ is a lower barrier for the integral curves from $B$ to $D$. We will prove it by contradiction.
We first show that $H_G(U)$ is a lower barrier for $U\in(1,1+\epsilon)$ where $\epsilon$ is sufficiently small. We compute 
\begin{align}\notag
	H_G'(U_B,H_B) = \frac{1-\mu}{1+k_2}<C_2(U_B,H_B).
\end{align}
It follows that the integral curve $H(U)$ starting at $B$ with the slope $C_2(U_B,H_B)$ and propagated by \eqref{ODE} to the right, lies above $H_G(U)$ for $U\in(1,1+\epsilon)$. Suppose there exists a $U_r\in(1,U_D)$ such that $H(U)$ first intersects $H_G(U)$ at $U=U_r$. Then, we have $H(U)>H_G(U)$ for any $U\in(1,U_r)$ and $\lim_{U\to U_r} H'(U) = +\infty$. On the other hand, we have $H_G'(U)\in(0,\infty)$ by Lemma \ref{U_G'(H)>0}, which is a contradiction. Hence, the integral curve $H(U)$ starting at $B$ satisfies $ H_G(U )< H(U) < H^{\text{sp}}(U)$ for $1<U< U_D$, and $H(U_D) = H_D$. 
\end{proof}

We conclude this section by establishing the properties of the solution curve near the point $B$ in terms of $y$.  
 Let $H^\text{BD}(y)$ and $U^\text{BD}(y)$ be corresponding to 
 the integral curves constructed in Lemma \ref{B to D}, and let $y_B$ be the corresponding $y$-value at the point $B$. 
\begin{lemma}\label{Asy at B}
	For any $\gamma \in (1,3)$ and $\mu \in (0,1)$, $y_B$ is a finite negative number. Moreover, for $y$ such that $0<y-y_B\ll 1$, the functions $H^\text{BD}(y)$ and $U^\text{BD}(y)$ satisfy
	\begin{equation}\label{Eq: Asy at B}
	\begin{aligned}
		U^\text{BD}(y) &= 1+ \frac{(\ga-1)(\mu-1)}{y_B} (y-y_B) + \sum_{n\geq 2}\tilde U_n (y-y_B)^n, \\
		H^\text{BD}(y) &= \frac{(1-\ga)(1+k_2)}{\ga y_B}(y-y_B) + \sum_{n\geq 2}\tilde H_n (y-y_B)^n.
	\end{aligned}
	\end{equation}
\end{lemma}
\begin{proof}
	The negativity of $y_B$ follows from the fact that there are no critical points between $D$ and $B$. To establish the finiteness of $y_B$, we introduce the variable $z = \ln (-y)$. The finiteness of $y_B$ is then equivalent to proving the finiteness of $z_B$. For notational simplicity, define the function	
	\begin{align}\label{4.26}
		\frac{d{U^\text{BD}}}{dz}(z)= \frac{G(U^\text{BD}(z),H(U^\text{BD}(z)))}{\Delta(U^\text{BD}(z),H(U^\text{BD}(z)))}=:\mathcal G(U^\text{BD}(z)).
	\end{align}
	From \eqref{yD} and \eqref{U'(yD)}, $z=z_D$ is a removable singularity of $\mathcal G(U^\text{BD}(z))$. Additionally, since the integral curve along $C_2(U_B,H_B)$ satisfies $H(U_B) = C_2(U_B,H_B)(U-U_B)+o(U-U_B)$, we obtain the following:  
	\begin{equation}\label{Eq: 1st de at B}
	\begin{aligned}
		\frac{d{H^\text{BD}}}{dz}(z_B) &=\lim_{U\to U_B} \frac{2C_2(U_B,H_B)(U-U_B)[C_2(U_B,H_B)(U-U_B)-(U_B^2-k_1U_B+\mu)]}{(U-1)^2-C_2(U_B,H_B)(U-U_B)}  = (\ga-1)(\mu-1),\\
		 \frac{d{U^\text{BD}}}{dz}(z_B)  &=\lim_{U\to U_B} \cfrac{C_2(U_B,H_B)(U-U_B)(U+k_2)-U(U-1)(U-\mu)}{(U-1)^2-C_2(U_B,H_B)(U-U_B)} = \frac{(1-\ga)(1+k_2)}{\ga}.
	\end{aligned}
	\end{equation}
	Consequently, $z=z_B$ is also a removable singularity of $\mathcal G(U^\text{BD}(z))$. 
	Furthermore, since $\mathcal G(U^\text{BD}(z))<0$ and continuous in the interval $z\in[z_B,z_D]$ 
	, there exist constants $-\infty< m<M<0$ such that
	\[ m< \mathcal G(U^\text{BD}(z))<M, \]
	for any $z\in[z_B,z_D]$. We then integrate the ODE $\frac{dz}{dU} = \frac{\Delta(U,H(U))}{G(U,H(U))}$ from $U_B$ to $U_D$ and obtain
	\begin{align*}
	 \frac{U_D-U_B}{m} >	z_D - z_B = \int_{U_B}^{U_D}  \frac{1}{\mathcal G(U)} d U > \frac{U_D-U_B}{M}.
	\end{align*}
	Therefore, $z_B$ is finite, which establishes the finiteness of $y_B$. The expansions for $H^\text{BD}(y)$ and $U^\text{BD}(y)$ follow directly from the smoothness of the curves given by Lemma \ref{der at B} and the derivatives computed in \eqref{Eq: 1st de at B}. This concludes the proof.
\end{proof}

\section{Proof of Theorem \ref{main theorem}} \label{Sec 5}

The existence of $\rho$ and $u$ solving the vacuum free boundary Euler  \eqref{Euler}, \eqref{BC1}-\eqref{IC} with the left vacuum boundary $b(t)$ follows from the existence of the solution $(U(y),H(y))$ to the ODE system \eqref{ode system} which connects $C\to A \to E \to D \to B$. For $t<0$, the condition $h(t,x)=0$ if and only if $x=0$, corresponding to the point $C$. For $t>0$, the moving boundary is given by $y\equiv y_B$, as $H(y_B) = 0$ uniquely determines the first vanishing point of $h(t,x)$ for any $t>0$. Thus, the moving boundary \eqref{b(t)} is 
satisfied.

Property 1 follows directly from the self-similar ansatz. Property 3 is a consequence of Property 2 and Lemma \ref{der at B}, while Property 4 follows from $H^\text{BD}(y_B)=0$ and ${H^\text{BD}}'(y_B) >0$ as in Lemma \ref{Asy at B}.  To establish Property 2, it is sufficient to verify the regularity of the original physical variables in neighborhoods of the points $A$, $E$ and $D$, as these are the only critical 
points relevant to the ODE system \eqref{ode system}. To avoid confusion, let $U^{\text{sp}}_{P,R}$ and $U^{\text{sp}}_{P,L}$ denote the components of the special solution $U^{\text{sp}}$ in the first and second quadrants of the $(U,H)$ plane near the point $P$, where $P\in\{A,E\}$. At  the point $A$, by setting $z = y^{-\mu}$,  \eqref{formula special} is transformed  into 
\begin{align*}
z = \frac{a_XU^{\text{sp}}_{A,X}(z)}{K^\mu (U_C - U^{\text{sp}}_{A,X}(z))^{1-\mu}},
\end{align*}
where $X\in\{L,R\}$, with $a_R = 1$ and $a_L=-1$. It then follows that 
\begin{align*}
	U^\text{sp}_{A,R}(z) &=K^\mu U_C^{1-\mu}z+\sum_{n\geq 2} U_n z^n, \quad z\to 0,\\
	U^\text{sp}_{A,L}(z) &=-K^\mu U_C^{1-\mu}z+\sum_{n\geq 2} (-1)^{n}U_n z^n, \quad z\to 0.
\end{align*}
Substituting these expressions into the self-similar ansatz \eqref{2.14} and \eqref{2.15}, we obtain:
\begin{align*}
	&u(t,x) = -\delta K^\mu U_C^{1-\mu} x^{1-\mu}-\sum_{n\geq 2}\delta U_n x^{-n\mu+1}(-t)^{n-1},\quad \text{when }t<0\text{ and } x\gg |t|^\delta,\\
	&u(t,x) = -\delta K^\mu U_C^{1-\mu} x^{1-\mu}+\sum_{n\geq 2}(-1)^n\delta U_n x^{-n\mu+1}t^{n-1},\quad \text{when }t>0\text{ and } x\gg |t|^\delta.
\end{align*}
Hence, the solution is smooth 
in a neighborhood of any point in $\{(t,x)\mid |t|\ll 1\text{ and }x>0\}$. 

Near the point  $E$, we let $V = \frac{1}{U}$ and rewrite \eqref{formula special} as 
\begin{align*}
y= -KV^\text{sp}_{E,X} (1- \frac{2\mu}{\gamma+1}V^\text{sp}_{E,X} )^{\frac{1}{\mu}-1}\quad \text{for }y\sim y_E.
\end{align*}
By applying a similar argument as at the point $A$, we obtain 
\begin{align*}
	&V^\text{sp}_{E,L}(y) =\sum_{n\geq 1} V_n y^n, \quad y\to 0^+,\\
	&V^\text{sp}_{E,R}(y) =\sum_{n\geq 1} V_n y^n, \quad y\to 0^-,
\end{align*}
where $V_1\neq 0$, as indicated by \eqref{Uy0+}. 
Substituting these into the ansatz \eqref{2.15}, we derive
\begin{align*}
	u(t,x) = \frac{\delta x}{t} \frac{1}{V^\text{sp}_{E,X}(\frac{x}{t^\delta})}= \frac{\delta t^{\delta-1}}{ V_1  +\sum_{n\geq 2} V_n x^{n-1} t^{-(n-1)\delta}} \quad \text{ for }   |x| \leq \epsilon t^\delta  \ \text{and} \ t>0
\end{align*}
for sufficient small $\epsilon$.
This demonstrates that $u(t,x)$ is smooth 
 in the neighborhood of every point with $t>0$. Finally, the Lipschitz continuity at the point $D$ follows directly from Lemma \ref{Le: asy at D both sides}. 


\section*{Acknowledgments}
JJ and JL were supported in part by
the NSF grant DMS-2306910.

\end{document}